\documentclass{aic}

\usepackage{bbm}
\usepackage{amsmath, amsthm, amssymb}
\usepackage[shortlabels]{enumitem}
\usepackage{thmtools}
\usepackage{mathtools}

\declaretheorem[name=Lemma, numberwithin = section]{lemma}
\declaretheorem[name=Theorem,sibling = lemma]{theorem}

\declaretheorem[name=Claim]{claim}

\usepackage[noabbrev,capitalise,nameinlink]{cleveref}
\crefname{claim}{Claim}{Claims}
\crefname{lemma}{Lemma}{Lemmas}
\crefname{theorem}{Theorem}{Theorems}
\crefname{proposition}{Proposition}{Propositions}
\crefname{question}{Question}{Questions}
\crefname{definition}{Definition}{Definitions}
\crefname{conjecture}{Conjecture}{Conjectures}
\crefname{corr}{Corollary}{Corollaries}
\crefformat{equation}{(#2#1#3)}
\Crefformat{equation}{Equation #2(#1)#3}
\crefname{remark}{Remark}{Remarks}

\newenvironment{subproof}[1][\proofname]{%
  \begin{proof}[#1]%
}{%
  \end{proof}%
}

\newcommand{\C}{\mathcal{C}}
\newcommand{\ex}{\text{ex}}
\aicAUTHORdetails{%
  title = {On polynomial degree-boundedness}, 
  author = {Romain Bourneuf, Matija Buci\'c, Linda Cook, and James Davies},
  plaintextauthor = {Romain Bourneuf, Matija Bucic, Linda Cook, James Davies},
  keywords = {extremal graph theory, graph coloring, $\chi$-boundedness, induced K\H{o}v\'ari-S\'os-Tur\'an~theorem, forbidden induced subgraph},
}   

\aicEDITORdetails{%
   year={2024},
   number={5},
   received={28 December 2023},   
   published={9 October 2024},  
   doi={10.19086/aic.2024.5},      
}   

\begin{document}

\begin{frontmatter}[classification=text]

\title{On Polynomial Degree-Boundedness} 

\author[rb]{Romain Bourneuf }
\author[mb]{Matija Buci\'c \thanks{Research supported in part by an NSF Grant DMS--2349013.}}
\author[lc]{Linda Cook}

\author[jd]{James Davies}

\begin{abstract}
 We prove a conjecture of Bonamy, Bousquet, Pilipczuk, Rzążewski, Thomassé, and Walczak, that for every graph $H$, there is a polynomial $p$ such that for every positive integer $s$, every graph of average degree at least $p(s)$ contains either $K_{s,s}$ as a subgraph or contains an induced subdivision of $H$.
    This improves upon a result of Kühn and Osthus from 2004 who proved it for graphs whose average degree is at least triply exponential in $s$ and a recent result of Du, Gir\~{a}o, Hunter, McCarty and Scott for graphs with average degree at least singly exponential in $s$.

     As an application, we prove that the class of graphs that do not contain an induced subdivision of $K_{r,t}$ is polynomially $\chi$-bounded. In the case of $K_{2,3}$, this is the class of induced theta-free graphs, and answers a question of Davies.
     Along the way, we also answer a recent question of McCarty, by showing that if $\mathcal{G}$ is a hereditary class of graphs for which there is a polynomial $p$ such that every \emph{bipartite} $K_{s,s}$-subgraph-free graph in $\mathcal{G}$ has average degree at most $p(s)$, then more generally, there is a polynomial $p'$ such that every $K_{s,s}$-subgraph-free graph in $\mathcal{G}$ has average degree at most $p'(s)$.
     Our main new tool is an induced variant of the K\H{o}v\'ari-S\'os-Tur\'an theorem, which we find to be of independent interest. 
\end{abstract}
\end{frontmatter}

\section{Introduction}

For graphs $G,H$, we say that $G$ is \emph{$H$-subgraph-free} if $G$ does not contain $H$ as a (not necessarily induced) subgraph. 
A \emph{subdivision} of a graph $H$ is any graph obtained from $H$ by replacing edges by (possibly) longer paths. 
In 2004, Kühn and Osthus \cite{Kuhn_Osthus00} proved that for every graph $H$, there is a function $f$ such that every $K_{s,s}$-subgraph-free graph with no induced subdivision of $H$ has average degree at most $f(s)$.
Their function $f$ was a triple exponential in $s$ and they asked for improved bounds in \cite{Kuhn_Osthus00}. 
Recently, this was improved to a single exponential bound by Du, Gir\~{a}o, Hunter, McCarty and Scott \cite{du2023induced}.
In 2022, Bonamy, Bousquet, Pilipczuk, Rzążewski, Thomassé, and Walczak \cite{BONAMY2022353} conjectured that $f$ can be taken to be a polynomial in $s$.
In the same paper, Bonamy, Bousquet, Pilipczuk, Rzążewski, Thomassé, and Walczak \cite{BONAMY2022353} proved their conjecture when $H$ is a path or a cycle.
Subsequently, Scott, Seymour, and Spirkl \cite{Scottpoly1} proved the conjecture in the case where $H$ is a tree.
The main result of this paper is a full proof of the conjecture of Bonamy, Bousquet, Pilipczuk, Rzążewski, Thomassé, and Walczak \cite{BONAMY2022353}. 

 \begin{theorem}\label{thm:poly-kuhn-osthus}
     For every graph $H$, there is a polynomial $p$, such that every $K_{s,s}$-subgraph-free graph with no induced subdivision of $H$ has average degree at most $p(s)$.
 \end{theorem}

\cref{thm:poly-kuhn-osthus} has an application to polynomial $\chi$-boundedness.
A class of graphs $\mathcal{G}$ is (polynomially) $\chi$-bounded if there exists a (polynomial) function $f: \mathbb{N} \to \mathbb{N}$ such that $\chi(G)\le f(\omega(G))$ for every $G\in \mathcal{G}$ (where $\chi(G)$ and $\omega(G)$ denote respectively the chromatic number and the clique number of $G$).
Note that, in a recent breakthrough Bria\'nski, the fourth author, and Walczak \cite{BDW2022} proved that there are $\chi$-bounded classes of graphs that are not polynomially $\chi$-bounded using a construction of Carbonero, Hompe, Moore and Spirkl \cite{triangle-free-construction}. 
Scott \cite{Scott97} conjectured that for every graph $H$, the class of graphs with no induced subdivision of $H$ is $\chi$-bounded.
This conjecture is known to be true when $H$ is a tree \cite{Scott97}, a cycle \cite{chudnovsky2017induced} or more generally a banana tree \cite{SCOTT2020487}.
However, Scott's \cite{Scott97} conjecture is false in general as shown by Pawlik, Kozik, Krawczyk, Laso{\'{n}}, Micek, Trotter, and Walczak \cite{PKKLMTW14}, and now many counter-examples are known \cite{chalopin2016restricted,POURNAJAFI2024103850}, including rather small graphs such as $K_5$ \cite{POURNAJAFI2024103850}.
Kühn and Osthus's \cite{Kuhn_Osthus00} theorem implies that the class of graphs not containing an induced subdivision of $K_{r,t}$ is $\chi$-bounded. Similarly, \cref{thm:poly-kuhn-osthus} implies polynomial $\chi$-boundedness for this class.

\begin{restatable}{theorem}{polychibdd}\label{thm:chi-bound}
    For every pair of positive integers $r,t$, the class of graphs not containing an induced subdivision of $K_{r,t}$ is polynomially $\chi$-bounded.
\end{restatable}

Of particular interest is the case $s=2$, $t=3$. A graph is a \emph{theta} if it is a subdivision of $K_{2,3}$, and a graph is \emph{induced theta-free} if it contains no induced theta. \cref{thm:chi-bound}, implies that induced theta-free graphs are polynomially $\chi$-bounded, which answers a question raised by Davies \cite{davieswheel}.
See \cite{SS2018} for a wonderful 2020 survey on $\chi$-boundedness by Scott and Seymour.

McCarty \cite{mccarty2021} proved that if $\mathcal{G}$ is a hereditary class of bipartite graphs whose $C_4$-subgraph-free graphs have bounded average degree, then there is a function $f$ such that every $K_{s,s}$-subgraph-free graph in $\mathcal{G}$ has average degree at most $f(s)$.
Along the way to proving \cref{thm:poly-kuhn-osthus}, we answer a recent question of McCarty \cite{rose-matrix} by showing the following:
\begin{theorem}\label{thm:McCarty}
    Let $\mathcal{G}$ be a hereditary class of graphs for which there is a polynomial $p$ such that every \emph{bipartite} $K_{s,s}$-subgraph-free graph in $\mathcal{G}$ has average degree at most $p(s)$.
    Then, there is a polynomial $p'$ such that every $K_{s,s}$-subgraph-free graph in $\mathcal{G}$ has average degree at most $p'(s)$.
\end{theorem}

Our main new tool for proving \cref{thm:poly-kuhn-osthus} and \cref{thm:McCarty} is an extension of an induced variant of the classical K\H{o}v\'ari-S\'os-Tur\'an theorem \cite{KST54} which was recently proved by Loh, Tait, Timmons and Zhou~\cite{LTTZ18}. 

The forbidden subgraph problem is a central problem in extremal combinatorics, and it has been widely studied. 
It can be described as follows: Given a graph $H$, find the maximum number of edges $\ex(n, H)$ an $n$-vertex graph can have without containing $H$ as a subgraph. 
The classical Erd\H{o}s-Stone-Simonovits theorem \cite{ES46} states that the asymptotic behavior of $\ex(n, H)$ is determined by $\chi(H)$ as follows: $\ex(n, H) = \left(1-\frac{1}{\chi(H)-1}\right) \binom{n}{2} + o(n^2)$. 
This solves the forbidden subgraph problem except for bipartite graphs, for which it implies $\ex(n, H) = o(n^2)$. 
The K\H{o}v\'ari-S\'os-Tur\'an theorem \cite{KST54} gives a more precise upper bound on $\ex(n, H)$ for bipartite $H$, namely $\ex(n, K_{s, s}) \leq c_{s}n^{2-1/s}$.

It is very natural to ask what happens in the forbidden subgraph problem if we instead forbid $H$ as an \emph{induced} subgraph. Unfortunately, in this case, the problem becomes trivial as soon as $H$ is not a complete graph: the maximum number of edges in an $n$-vertex graph without $H$ as an induced subgraph is simply $\binom{n}{2}$. 
 The problem however becomes interesting again when we forbid simultaneously some graph $F$ as a subgraph and another graph $H$ as an induced subgraph. 
This problem was introduced by Loh, Tait, Timmons and Zhou \cite{LTTZ18}, who defined $\ex(n, F, H\text{-ind})$ as the maximum number of edges an $n$-vertex graph can have without containing $F$ as a subgraph or $H$ as an induced subgraph. The main idea behind this question is that forbidding a graph $H$ as an induced subgraph is (in many cases) a much weaker restriction than forbidding it as a subgraph and yet we can recover bounds similar to the K\H{o}v\'ari-S\'os-Tur\'an bounds on the extremal number, so long as we forbid an additional graph $F$ as a subgraph. Note that this is mainly of interest when $\ex(n,H) \ll \ex(n,F)$.
Loh, Tait, Timmons and Zhou \cite{LTTZ18} proved that if we forbid \emph{any} graph $F$ as a subgraph and $K_{r, t}$ as an \emph{induced} subgraph, then we can recover asymptotically the same upper bound as in the K\H{o}v\'ari-S\'os-Tur\'an theorem, namely they proved that $\ex(n, F, K_{r, t}\text{-ind}) = O(n^{2-1/r})$ where the constant depends only on $F$ and $s$.

We prove the following variant where $K_{s,s}$ is forbidden as a subgraph, and an \emph{arbitrary} bipartite graph $H$ is forbidden as an induced subgraph.

\begin{theorem}\label{thm:induced-turan-type}
    Let $H$ be a bipartite graph. 
    There exist constants $0 < \varepsilon_H \leq \frac{1}{2}$ and $c_H > 1,$ depending only on $H$, such that for every $s \geq 1$, $\ex(n, K_{s, s}, H\text{-ind}) \leq c_Hs^{4}n^{2-\varepsilon_H}$.
\end{theorem}

We remark that since we require $H$ to be an \emph{induced} subgraph instead of just a subgraph, this does not follow immediately from the result for $H=K_{r,t}$. In addition, in the case where $H = K_{r, t}$, a key property is that in a bipartite graph forbidding a $K_{r,t}$ as an induced subgraph is equivalent to simply forbidding it as a subgraph, so one may apply the classical Tur\'an results. One does not in general have access to this powerful observation when $H$ is not complete bipartite and we employ significantly different ideas to prove the above result. In particular, we use a density increment approach recently introduced by the second author, Nguyen, Scott and Seymour \cite{loglogEH} in proving the best-known general bound on the infamous Erd\H{o}s-Hajnal conjecture. 

We also note that if $H$ is not bipartite one may simply take the extremal construction for $\ex(n,K_{s,s})$ and pass to a bipartite subgraph with the most edges, losing at most half of the edges while ensuring there is no copy of $H$ (induced or otherwise), so we have no gain in the asymptotic behaviour from this assumption. Furthermore, if we forbid a non-bipartite graph as a (not necessarily induced) subgraph, by considering a balanced complete bipartite graph it is easy to see that the extremal number is still quadratic unless $H$ is a complete bipartite graph.

\textbf{Remark.} We note that since this paper first appeared on arXiv \Cref{thm:induced-turan-type} found several interesting applications. For example, given a fixed $p \in (0,1)$, the binomial random graph $G=\mathcal{G}(n,p)$ and a hereditary property $\mathcal{P}$ such that there is a bipartite graph $H \notin \mathcal{P}$, it holds with high probability that any subgraph of $G$ that belongs to $\mathcal{P}$ has at most $n^{2-\varepsilon_{\mathcal{P}}}$ edges. This follows essentially immediately from \Cref{thm:induced-turan-type} combined with the observation that with high probability $\mathcal{G}(n,p)$ does not contain a $K_{s,s}$ as a (not necessarily induced) subgraph with $s=2\log n$. This answers a question raised by Alon, Krivelevich, and Samotij in \cite{random-turan}. We note that this has also been proved by Fox, Nenadov, and Pham \cite{jacob-paper} who prove much stronger and more general results in this direction, and independently by Clifton, Liu, Mattos, and Zheng \cite{hong-paper} who use the ideas behind our proof of \Cref{thm:induced-turan-type}. In addition, a much more precise version (for specific forbidden bipartite $H$) has been used by Milojevi\'c, Sudakov, and Tomon to prove a number of very interesting results about point-hyperplane incidences \cite{istvan-paper}. We note that this paper is using ideas from extremal graph theory similar to ours to improve upon more geometric ideas used by Fox, Pach, Sheffer, Suk, and Zahl in \cite{semi-algebraic} to improve the classical bounds in the Zarankiewicz problem for semi-algebraic ground graphs. We point an interested reader to \cite{istvan-paper} for more details on this topic. 
For several other applications, strengthenings and extensions see \cite{benny-paper}.
See also \cite{du2024survey} for a new recent survey on degree-boundedness by Du and McCarty.

\textbf{Remark.} Results similar to ours were obtained independently and around the same time by Gir\~{a}o and Hunter \cite{girao2023induced}.
While the general approach seems similar between the two papers, there are also significant differences, for example, in how we prove \Cref{thm:induced-turan-type} compared to their proof of an analogous result. 

\textbf{Organisation.}
In \cref{sec:turan}, we prove our induced variant of the K\H{o}v\'ari-S\'os-Tur\'an~theorem (\cref{thm:induced-turan-type}), our main tool.
Then in \cref{sec:main}, we use \cref{thm:induced-turan-type} to prove a technical strengthening of \cref{thm:McCarty} (see \cref{thm:bipartite-degree-boundedness}).
Lastly, in \cref{sec:consequence}, we use \cref{thm:bipartite-degree-boundedness} to prove \cref{thm:poly-kuhn-osthus}, which in turn is used to quickly prove \cref{thm:chi-bound}.

\textbf{Notation}.
Given a graph $G$, we denote its number of edges by $e(G)$ and its number of vertices by $|G|$.
For every set $X \subseteq V(G)$ we denote by $G[X]$ the graph induced by $X$.
For every disjoint $A, B \subseteq V(G)$ we define $G[A, B]$ to be the bipartite graph obtained from $G[A \cup B]$ by deleting all edges with both ends in $A$ or with both ends in $B$.
We denote by $e(A, B)$ the number of edges with one endpoint in $A$ and the other endpoint in $B$. Hence, $e(A, B) = e(G[A, B])$.
The degree of a vertex $v \in V(G)$ is denoted by $d(v)$ and for any subset $X \subseteq V(G)$, we define $d_X(v) \coloneqq |N_G(v) \cap X|$.
A \emph{regular} graph is a graph in which every vertex has the same degree.
We denote the average degree of a graph $G$ by $d(G)$.
The \emph{maximum average degree} of a graph $G$ is the maximum of $d(H)$ taken over all (induced) subgraphs $H$ of $G$.
A \emph{proper subdivision} of $G$ is a subdivision of $G$ in which every edge is replaced with a path on at least three vertices.
A \emph{$1$-subdivision} of $G$ is a proper subdivision of $G$ in which every edge is replaced by a path on exactly three vertices.
All logarithms will be taken in base 2.
{For a bipartite graph $H$ we let $\varepsilon_H, c_H$ denote the constants $0 < \varepsilon_H \leq \frac{1}{2}$  and $c_H>1$ from \cref{thm:induced-turan-type}.}

\section{Induced extremal numbers}\label{sec:turan}

This section is dedicated to proving \cref{thm:induced-turan-type}.
The following is an immediate consequence of Theorem 1.1 from \cite{LTTZ18}.

\begin{lemma} \label{lem:induced-turan}
    Let $r, t \geq 2$. 
    There exists a constant $0< \varepsilon \leq \frac{1}{2}$ and a constant $c > 1$ such that for every $s \geq 1$, we have $\ex(n, K_{s, s}, K_{r, t}\text{-ind}) \leq cs^{4}n^{2-\varepsilon}$.
\end{lemma}

For our purposes, one should think of $K_{r,t}$ as small and fixed and $K_{s,s}$ as being large. In this regime, the above result is remarkable in the following sense. 
When we forbid $K_{r, t}$ as an induced subgraph, the maximum number of edges in an $n$-vertex $K_{s, s}$-subgraph-free graph is still subquadratic in $n$, but in this case, the dependency on $s$ is polynomial.
In contrast, the K\H{o}v\'ari-S\'os-Tur\'an theorem gives a bound on $\ex(n, K_{s,s})$ with exponential dependency on $s$. The goal of this section is to extend the above result to allow forbidding an arbitrary (small) bipartite graph $H$ as an induced subgraph, as opposed to simply complete bipartite ones. Namely, our goal is to prove \cref{thm:induced-turan-type}. 

This will follow by starting from the above result and repeatedly using the following lemma which shows that if \cref{thm:induced-turan-type} holds for some bipartite graph $H$ then it also holds (with weaker constants) for $H-e$, where $e = uv$ is an arbitrary edge of $H$. 
The basic idea behind the proof is that knowing that \cref{thm:induced-turan-type} holds for $H$ allows us to conclude a supersaturation result, namely, if we have density a bit higher than extremal then we can find many induced copies of $H$ (i.e. many induced subgraphs isomorphic to $H$) in any $K_{s,s}$-subgraph-free graph.
So, if we try to prove the result for $H-e$ by contradiction, we may assume our ground graph has many edges and thus we can conclude there are many induced copies of $H$.
To show this supersaturation result, we sample a uniformly random subset of vertices and then use martingale concentration inequalities to show that the subsampled graph still has high enough density to guarantee us an induced copy of $H$. Using this, we can guarantee at least $2sn^{|H|-1}$ induced copies of $H$ in our $n$ vertex graph. Thus, there is an (induced) embedding of $H - \{u, v\}$ which can be extended into an (induced) embedding of $H$ in at least $2sn$ different ways. In particular, we have sets $C_u$ and $C_v$ of at least $2s$ ``candidates'' where we could embed $u$ and $v$, respectively. If there is a non-edge between a vertex of $C_u$ and a vertex of $C_v$, this gives rise to an induced copy of $H-e$, while otherwise, we find a $K_{s,s}$, in either case we arrive at the desired contradiction.

\begin{lemma} \label{lem:edge-removal}
    Let $H$ be a bipartite graph. 
    If \cref{thm:induced-turan-type} is true for $H$, and $e \in E(H)$, then \cref{thm:induced-turan-type} is true for $H-e$. 
\end{lemma}

\begin{proof}
    Since $e \in E(H)$ we have $|H| \ge 2$.
    Let $0< \varepsilon \leq \frac{1}{2}$ and $c > 1$ be the constants such that for every $s \geq 1$, $\ex(n, K_{s, s}, H\text{-ind})$ $\leq cs^{4}n^{2-\varepsilon}$, as guaranteed by \cref{thm:induced-turan-type} applied to $H$.
    
    Set $\varepsilon' \coloneqq \frac{\varepsilon}{2|H|}$ 
    and 
    $c' \coloneqq \max\left\{|H|, c\right\}$.
    Let $s\geq 1$ be an integer, 
    $G$ be a $K_{s, s}$-subgraph-free graph and $n~=~|G|$.
    Suppose for a contradiction, that $G$ does not contain $H - e$ as an induced subgraph and that {$e(G)~>~c's^{4}n^{2-\varepsilon'}$}.
    Then, $n^2 \geq e(G) > c's^{4}n^{2-\varepsilon'},$ and hence 
    \begin{equation}\label{eq:n-lwr-bnd}
        n \geq (c's^{4})^{1/\varepsilon'} \geq \max\left\{ 4^{1/\varepsilon'} \cdot 16s^2, (4|H|)^{|H|} \cdot 4s\right\},
    \end{equation}
    where we used $\varepsilon' \leq 1/|H|$ and $s \ge 2$, which must hold since $G$ is $K_{s,s}$-subgraph-free and $e(G)>0$.

    \begin{claim} \label{cl:density-subsets}
    Let $m$ be an integer
    such that $m \geq \left(\frac{n}{4s}\right)^{1/|H|}$ and $m \leq \sqrt{n}$. 
    If $U$ is a subset of $m$ vertices of $G$, chosen uniformly at random, then $$\mathbb{P}\left[ e(G[U]) > cs^{4}m^{2-\varepsilon} \right] \geq 3/4.$$
    \end{claim}

    \begin{subproof}
        Let $X_1, \ldots, X_m$ be independent random variables, each with uniform distribution on $V(G)$. 
        Let $f(X_1, \ldots, X_m)$ be the number of edges of $G[\{X_1, \ldots, X_m\}]$. Note that it is possible that $X_i = X_j$ when $i \neq j$, hence $\{X_1, \ldots, X_m\}$ is a set of \emph{at most} $m$ vertices. 
        Changing the outcome of a single $X_i$ changes the value of $f$ by at most $m-1$. In particular, $f$ satisfies the ``bounded differences property'' with bound $m-1$. 
        We have 
        \begin{align*}
            \mathbb{E}[f(X_1, \ldots, X_m)] &= \mathbb{E}\left[\sum_{\{i, j\} \in \binom{[m]}{2}} \mathbbm{1}_{X_iX_j \in E(G)}\right]
                                            = \sum_{\{i, j\} \in \binom{[m]}{2}} \mathbb{P}\left[X_iX_j \in E(G)\right] 
                                            = \binom{m}{2}\frac{2e(G)}{n^2}
                                            \geq 2c's^{4}\binom{m}{2}n^{-\varepsilon'}.
        \end{align*} 
        We have $m^{\varepsilon} \ge \left(\frac{n}{4s}\right)^{{\varepsilon}/{|H|}}=\left(\frac{n}{4s}\right)^{2\varepsilon'}\ge 4n^{\varepsilon'},$ where the last inequality follows from \eqref{eq:n-lwr-bnd}.
         Hence, $cs^{4}m^{2-\varepsilon} \leq c's^{4}\binom{m}{2}n^{-\varepsilon'} \le \frac12 \mathbb{E}[f(X_1, \ldots, X_m)]$. Thus, by McDiarmid's inequality \cite[Theorem 2.7]{McDiarmid1998}:
        \begin{align*}
            \mathbb{P}\left[f(X_1, \ldots, X_m) \leq cs^{4}m^{2-\varepsilon}\right] 
                &\leq \mathbb{P}\left[f(X_1, \ldots, X_m) \leq \frac{1}{2}\mathbb{E}[f(X_1, \ldots, X_m)]\right] \\
                &= \mathbb{P}\left[f(X_1, \ldots, X_m) - \mathbb{E}[f(X_1, \ldots, X_m)] \leq -\frac{1}{2}\mathbb{E}[f(X_1, \ldots, X_m)]\right] \\
                &\leq \exp\left[\frac{-\mathbb{E}[f(X_1, \ldots, X_m)]^2}{2m(m-1)^2}\right] \leq \exp\left[\frac{-\left(2cs^{4}m^{2-\varepsilon}\right)^2}{2m^3}\right] \le e^{-4m^{1-2\varepsilon}}
                \leq \frac{1}{8},
        \end{align*}
        where the last two inequalities use $c> 1, s \ge 2$ and $\varepsilon \le \frac 12$.
        
        Observe that conditioning on the event that all $X_i$ are pairwise distinct, we obtain that $\{X_1, \ldots, X_m\}$ is a uniformly random subset of $m$ vertices of $G$.
        Note also that $\mathbb{P}[\exists i \neq j, X_i = X_j] \leq \sum_{\{i, j\} \in \binom{[m]}{2}} \mathbb{P}[X_i = X_j] = \binom{m}{2} \frac{1}{n} \leq \frac{m^2}{2n} \leq \frac{1}{2}$.
    
        Thus, $\mathbb{P}[\forall i \neq j, X_i \neq X_j] \geq 1/2$. Hence, 
        \begin{align*}
            \mathbb{P}\left[e(G[U]) \leq cs^{4}m^{2-\varepsilon}\right] &= \mathbb{P}\left[f(X_1, \ldots, X_m) \leq cs^{4}m^{2-\varepsilon} \mid \forall i \neq j, X_i \neq X_j\right] \\
                &= \frac{\mathbb{P}\left[f(X_1, \ldots, X_m) \leq cs^{4}m^{2-\varepsilon} \cap \forall i \neq j, X_i \neq X_j\right]}{\mathbb{P}\left[\forall i \neq j, X_i \neq X_j\right]} \\
                &\leq 2\mathbb{P}\left[f(X_1, \ldots, X_m) \leq cs^{4}m^{2-\varepsilon}\right] \leq \frac{1}{4}.
        \end{align*}
\vspace{-1.4cm}
        
    \end{subproof}

    Let $m = \lfloor\left(\frac{3n}{8s}\right)^{1/|H|}\rfloor$.
    
    Note that $\left(\frac{3n}{8s}\right)^{1/|H|} =  \left (\frac32\right)^{1/|H|}\cdot \left(\frac{n}{4s}\right)^{1/|H|} \ge \left(1+\frac1{4|H|}\right)\cdot \left(\frac{n}{4s}\right)^{1/|H|}\ge \left(\frac{n}{4s}\right)^{1/|H|} + 1$, where we used the inequality $(3/2)^x \ge 1+x/4$ which holds for any $x\ge 0$ and \eqref{eq:n-lwr-bnd}. 
    So $m \geq \left(\frac{n}{4s}\right)^{1/|H|}$. 
    Moreover, since $|H| \ge 2$ we have ${m \leq \left(\frac{3n}{8s}\right)^{1/|H|} \leq \sqrt{n}}$.
    By \cref{cl:density-subsets}, at least $\frac{3}{4}\binom{n}{m}$ of the $m$-vertex subsets of $G$ induce more than $cs^{4}m^{2-\varepsilon}$ edges. Let $X$ be such a subset. Then, since $G[X]$ has no $K_{s, s}$ subgraph and by assumption \cref{thm:induced-turan-type} holds for $H$, we obtain that $G[X]$ contains an induced copy of $H$. 
    
    Hence, there are at least $\frac{3}{4}\binom{n}{m}$ subsets $X$ of size $m$ such that $G[X]$ contains an induced $H$. Conversely, every induced copy of $H$ is contained in exactly $\binom{n-|H|}{m-|H|}$ $m$-vertex subsets of $G$.
    By double counting the pairs $(Y, X)$ such that $|X| = m, Y \subseteq X$ and $G[Y] \cong H$, using that $m \leq \left(\frac{3n}{8s}\right)^{1/|H|}$ for the last inequality, we get that the number $\eta_H$ of distinct induced copies of $H$ satisfies: \begin{align*}
        \eta_H &\geq \frac{3}{4} \frac{\binom{n}{m}}{\binom{n-|H|}{m-|H|}} = \frac{3}{4} \frac{n(n-1)\ldots(n-|H|+1)}{m(m-1)\ldots(m-|H|+1)} \geq \frac{3}{4} \left(\frac{n}{m}\right)^{|H|}
             \geq 2sn^{|H|-1}. 
    \end{align*}
    Recall $e \in E(H)$ and let $u, v$ denote the endpoints of $e$. Since there are at least $2sn^{|H|-1}$ pairwise distinct induced copies of $H$ in $G$, there is a set $X \subseteq V(G)$ of size $|H|-2$ such that $G[X] \cong H - \{u, v\}$ and such that there are at least $2sn$ different ways to extend $X$ to a set $Y$ of size $|H|$ such that $G[Y] \cong H$. 
    Let $C_u \subseteq V(G) \setminus X$ be the set of vertices that can play the role of $u$ in one of these extensions of $X$, and define $C_v$ similarly. The number of edges between $C_u$ and $C_v$ is equal to the number of extensions of $X$ into a set $Y$ such that $G[Y] \cong H$. Thus, there are at least $2sn$ edges between $C_u$ and $C_v$. In particular, since $|C_u|, |C_v| \leq n$, we have $|C_u|, |C_v| \geq 2s$. Take an arbitrary subset $U$ of $C_u$ of size $s$, and then an arbitrary subset $V$ of $C_v \setminus U$ of size $s$. If some vertex $x_u \in U$ is not adjacent to some vertex $x_v \in V$, then $G[X \cup \{x_u, x_v\}] \cong H-e$, which is impossible by assumption. Thus, $U$ is complete to $V$ and therefore $G$ has a $K_{s, s}$ subgraph, a contradiction.
\end{proof}

\textbf{Remark.} 
For our applications in this paper, it was enough to show that \Cref{thm:induced-turan-type} holds with some $\epsilon_H >0$, and we made no attempts to optimize the value of $\epsilon_H$. Since this paper appeared on arXiv, improved bounds on $\epsilon_H$ have been established, including \cite{benny-paper}, which gives essentially optimal bounds on $\epsilon_H$. 

\section{Bipartite reduction}\label{sec:main}

In this section, we shall prove a technical strengthening of \cref{thm:McCarty} (see \cref{thm:bipartite-degree-boundedness}), which shall then also be used in the final section to prove \cref{thm:poly-kuhn-osthus} and \cref{thm:chi-bound}, the main results of this paper.

The following ``regularisation'' lemma will allow us to pass to either an almost regular subgraph with similar average degree or to a very unbalanced bipartite subgraph accounting for a constant fraction of the edges. The following is a slight weakening of \cite[Lemma 3.4]{du2023induced}, which in turn builds on \cite{erdos-sauer} and a long history of similar regularisation lemmas. Let us highlight here a remarkable recent breakthrough of Janzer and Sudakov \cite{erdos-sauer}, who resolved the Erd\H{o}s-Sauer problem on minimum average degree requirement to guarantee existence of a regular subgraph.
\begin{lemma}
\label{lem:regular-or-bipartite}
There exists a function $f_1$ such that the following holds. 
Let $\gamma \in (0, 1/5)$ and $d \geq f_1(\gamma)$.
Let $G_0$ be an $n$-vertex graph which is $d$-degenerate and has average degree $d$. 
Then, either
\begin{enumerate}[(A)]
    \item $G_0$ contains an induced subgraph $G^*$ of average degree at least $6d^{1-5\gamma}$ and maximum degree $\Delta(G^*) \leq 6d^{1+3\gamma}$, or
    \label{regular}
    \item there is a partition $V(G_0) = A \sqcup B$ such that $e(A, B) \geq nd/4$ and $|A| \geq 2^{d^{\gamma}-2}|B|$.
    \label{bipartite}
\end{enumerate}
\end{lemma}

The following lemma will be an initial step in resolving the ``almost regular'' outcome of \Cref{lem:regular-or-bipartite}, namely case \ref{regular}. It is a modified version of Lemma 3.3 from \cite{du2023induced}. 
It says that any almost regular graph $G$ with average degree at least polynomial in $s$, which does not have $K_{s,s}$ as a subgraph or $H$ as an induced subgraph contains an induced subgraph with girth at least five and only polynomially smaller average degree. 
The proof is a basic application of the alteration method, where we subsample every vertex with some probability and then remove any still contained in a $C_3$ or $C_4$.
Recall, $\varepsilon_H$ and $c_H$ are the constants we obtain by applying \cref{thm:induced-turan-type} to $H$.

\begin{lemma} \label{lem:balanced-case}
    Let $H$ be a bipartite graph. There exists a polynomial $P$ satisfying the following. Let $s \geq 2$ and $\Delta \geq P(s)$.
    Let $G$ be a $K_{s, s}$-subgraph-free graph and no $H$ as an induced subgraph, with maximum degree at most $\Delta$ and average degree at least $\Delta^{1-\varepsilon_H/10}$. Then, $G$ has an induced subgraph $G'$ of girth at least five with average degree $d(G') \ge \frac{1}{2}{\Delta^{\varepsilon_H/10}}$.
\end{lemma}

\begin{proof}
   Let $\varepsilon=\varepsilon_H$ and $c=c_H$.
    Set $P(x) = x^{20/\varepsilon}+(116c)^{10/\varepsilon}$.
    For $\ell \in \{3,4\}$, let $\mathcal{C}_\ell$ denote the set of $\ell$-cycles of $G$ and let $\mathcal{S_\ell}$ denote the set of pairs $(e, C)$ with $e \in E(G)$, $C \in \C_\ell$ such that $e$ contains exactly one vertex of $C$.
 
    Let $xy$ be an edge of $G$. Then, $|N(x) \cup N(y)| \leq 2\Delta$ so by \cref{thm:induced-turan-type}, $G[N(x) \cup N(y)]$ has at most $cs^4(2\Delta)^{2-\varepsilon}$ edges. 
    There are at most two ways an edge of $G[N(x) \cup N(y)]$ can be extended to a four cycle using the edge $xy$. 
    Therefore, the number of 4-cycles containing the edge $xy$ is at most $8cs^4\Delta^{2-\varepsilon}$.
    Since $G$ has at most $n\Delta$ edges, it follows that $|\mathcal{C}_4| \leq 2cs^4n\Delta^{3-\varepsilon}$.
    For every cycle $C$, there are at most $4\Delta$ edges $e$ such that $(e, C) \in \mathcal{S}_4$ since every vertex of $C$ has degree at most $\Delta$. 
    Therefore, $|\mathcal{S}_4| \leq 8cs^4n\Delta^{4-\varepsilon}$. 
    Similarly, every vertex $x$ is in at most $cs^4\Delta^{2-\varepsilon}$ triangles so $|\mathcal{C}_3| \leq cs^4n\Delta^{2-\varepsilon}$ and $|\mathcal{S}_3| \leq 3cs^4n\Delta^{3-\varepsilon}$.

    Set $p \coloneqq \Delta^{\varepsilon/5-1}$. Let $U$ be a random subset of $V(G)$ containing each vertex independently with probability $p$. Let $\mathcal{C}$ be the set of 3-cycles and 4-cycles of $G[U]$, in particular $\mathcal{C} \subseteq  \mathcal{C}_3 \cup \mathcal{C}_4$. 
    Let $U' \subseteq U$ be the set of vertices that are in a 3-cycle or 4-cycle in $G[U]$.
    Let $G'$ be the graph induced by $U \setminus U'$. Then, $G'$ has girth at least 5.

    Let $X = e(G[U]), Y = |\mathcal{C}|$ and let $Z$ denote the number of pairs $(e, C) \in \mathcal{S}_3 \cup \mathcal{S}_4$ such that both $e \subseteq U$ and $C \in \mathcal{C}.$
    Let us examine the edges we delete from $G[U]$ when we remove $U'$.
    For each $C \in \mathcal{C}$, when we delete $V(C)$ from $G[U]$ we potentially remove six edges (between vertices of $C$) and we remove all edges $e$ which intersect $C$ at exactly one vertex (i.e. we remove all edges $e$ for which $(e, C) \in \mathcal{S}_3 \cup \mathcal{S}_4$).
    Overall, we get $e(G') \geq X - 6Y - Z$. 
    We now have:
    \begin{align*}
        \mathbb{E}[X] &= p^2e(G) = \frac{p^2}{2}nd(G).\\
        \mathbb{E}[Y] &= p^4|\mathcal{C}_4| + p^3|\mathcal{C}_3| \leq 2cp^4s^4n\Delta^{3 - \varepsilon} + cp^3s^4n\Delta^{2-\varepsilon} = cs^4np(2\Delta^{-2\varepsilon/5} + \Delta^{-3\varepsilon/5}) \le 3cs^4np\Delta^{-2\varepsilon/5}.\\
        \mathbb{E}[Z] &= p^5|\mathcal{S}_4| + p^4|\mathcal{S}_3| \leq 8cp^5s^4n\Delta^{4-\varepsilon} + 3cp^4s^4n\Delta^{3 - \varepsilon}= cs^4np(8\Delta^{-\varepsilon/5} + 3\Delta^{-2\varepsilon/5}) \le 11cs^4np\Delta^{-\varepsilon/5}.
    \end{align*} 

    Thus, \begin{align*}
        \mathbb{E}\left[e(G') - |U| \cdot \frac{pd(G)}{4}\right] &\ge \mathbb{E}\left[X - 6Y - Z - |U| \cdot \frac{pd(G)}{4}\right] \\
            &\geq \frac{p^2}{2}nd(G)  - 18cs^4np\Delta^{-2\varepsilon/5} - 11cs^4np\Delta^{-\varepsilon/5} - np \cdot \frac{pd(G)}{4}\\
            &\ge np\left(\frac{pd(G)}{4} - 29cs^4\Delta^{-\varepsilon/5}\right) > 0,
    \end{align*}
where the final inequality holds since $pd(G)\cdot \Delta^{\varepsilon/5}\ge \Delta^{\varepsilon/10}\cdot P(s)^{\varepsilon/5} > 116c\cdot s^4.$ 

Choosing an outcome $U$ such that the above holds, we have $e(G') \geq \frac{pd(G)}{4}|U| > 0$ so $G'$ has average degree at least ${\frac12{pd(G)}\ge \frac{1}{2}\Delta^{\varepsilon/10}}.$
\end{proof}

We can further extend \cref{lem:balanced-case} to obtain a bipartite induced subgraph with large average degree. For this, we require the following nice result of Kwan, Letzter, Sudakov, and Tran \cite{kwan2020}.
\begin{theorem}[\protect{Theorem 1.2 of \cite{kwan2020}}] \label{lem:bipartite-reduction}
    There exists a constant $C>0$ such that every triangle-free graph with average degree at least $d$ contains an induced bipartite subgraph of average degree at least $\frac{C\log d}{\log\log d}$.
\end{theorem}

By \cref{lem:balanced-case} (with $H$ chosen to be an arbitrary bipartite graph with average degree at least $k$) and \cref{lem:bipartite-reduction}, we obtain the following, which essentially handles the ``almost regular'' outcome of \Cref{lem:regular-or-bipartite}, namely case \ref{regular}.

\begin{lemma} \label{lem:balanced-case-to-bipartite}
    For every positive integer $k$, there exists $\varepsilon=\varepsilon(k) > 0$ and a polynomial $p_k$ such that the following holds. Let $s \geq 2$ be an integer and $\Delta \geq p_k(s)$.
    Let $G$ be a $K_{s,s}$-subgraph-free graph with maximum degree at most $\Delta$ and average degree at least $\Delta^{1-\varepsilon}$. Then, $G$ has a $C_4$-subgraph-free bipartite induced subgraph with average degree at least $k$.
\end{lemma}

\begin{proof}
    Let $H$ be some fixed bipartite $C_4$-subgraph-free graph with average degree at least $k$ (for instance the incidence graph of a projective plane of order at least $k$). Let $\varepsilon=\varepsilon_H/10$. 
    Let $C$ be the constant provided by \Cref{lem:bipartite-reduction} and let $p'$ be the polynomial provided by \cref{lem:balanced-case} for our choice of $H$. Let us set $p_k(x)= p'(x) + C'$, where $C'$ is a large enough constant to guarantee that $\displaystyle \frac{C\log \left(\Delta^{\varepsilon} / 2\right)}{\log\log \left( \Delta^{\varepsilon} / 2\right)}  \ge k$ for any $\Delta \ge C'$. 
    
    Suppose towards a contradiction that for some integer $s \geq 2$ there exists a graph $G$ which is $K_{s,s}$-subgraph-free, has maximum degree at most $\Delta \ge p_k(s)$, average degree at least $\Delta^{1-\varepsilon}$ and contains no $C_4$-subgraph-free bipartite induced subgraph with average degree at least $k$. 
    
    The last of these assumptions implies that $G$ does not contain $H$ as an induced subgraph, so we may apply \cref{lem:balanced-case} with this $H$ and our $s$ to obtain an induced subgraph $G'$ of $G$ with girth at least five and average degree at least $\frac{1}{2}\Delta^{\varepsilon}$. In particular, $G'$ is triangle-free. Applying \Cref{lem:bipartite-reduction} to $G'$, 
    we obtain an induced bipartite subgraph $G''$ of $G'$, and hence also of $G$, with average degree at least $\displaystyle \frac{C\log \left(\Delta^{\varepsilon} / 2\right)}{\log\log \left( \Delta^{\varepsilon} / 2\right)} \ge k$. Since $G''$ is an induced subgraph of $G'$, which has girth at least five, $G''$ is also $C_4$-subgraph-free, giving the desired contradiction.
\end{proof}

We will use the following lemma to deal with the very unbalanced outcome of \Cref{lem:regular-or-bipartite}, namely case \ref{bipartite}. It is a strengthening of Lemma 3.5 from \cite{du2023induced}, under the additional assumption of not having some bipartite graph $H$ as an induced subgraph.
It says that if we have a very unbalanced partition of the vertices of a $K_{s,s}$-subgraph-free, $d$-degenerate graph $G$ with no $H$ as an induced subgraph and with high density between the parts (in terms of $d$), then we can find an induced bipartite subgraph with still very unbalanced parts and in which every vertex on the larger side has degree equal to $p(d)$ where $p$ is a fixed polynomial.
The proof involves some simple cleaning of the graph, using degeneracy to ensure independence in the larger part and then subsampling the smaller part with low enough probability to ensure that it is unlikely that a subsampled vertex has any of its neighbours subsampled as well, and we simply delete any such vertices. This still allows us to show that, in expectation, a substantial proportion of the vertices in the larger part have the desired fixed degree to the smaller one.

\begin{lemma} \label{lem:bipartite-case}
    Let $H$ be a bipartite graph. 
    There exists a function $f_3$ such that the following holds.
    Let $s \geq 1, \delta \in (0, 1/2)$ and $d \geq f_3(\delta)$. Let $G_0$ be a $d$-degenerate graph with $V(G_0) = A \sqcup B$, that has no $K_{s, s}$-subgraph and no $H$ as an induced subgraph. Suppose that $|A| \geq 2^{d^{\delta}}|B|$ and $e(A, B) \geq 3\sqrt{d}|A|$.
    Then, for any $\displaystyle r \leq \min\Big\{d^{\delta/4}, \frac{d^{\varepsilon_H/2}}{4c_Hs^4}\Big\}$, we can find subsets $A' \subseteq A, B' \subseteq B$ that are both independent sets in $G_0$, with $|A'| \geq 2^{d^{\delta/2}}|B'|$ and $|N(a) \cap B'| = r$ for every $a \in A'$.
\end{lemma}

\begin{proof}
    We define $f_3$ to be a sufficiently fast-growing function to guarantee the following:
    for each $\delta \in (0, 1/2)$ and $d \ge f_3(\delta)$ all of the following conditions are true: 
    \begin{enumerate}[(i)]
        \item $d^\delta \geq 2d^{\delta/2} + \frac{3}{2}\log(d) + \log(20)$,
        \label{f3:log20}
        \item $d/2 \geq 10+d^{\delta/4}$, and
        \label{f3:halfd}
        \item $d^{\delta/2}\ge 2d^{\delta/4}\log(d) + 1$. 
        \label{f3:manyexponents}
    \end{enumerate}
    Let $s \geq 1$, $\delta \in (0, 1/2)$ and $d \ge f_3(\delta)$ be fixed and let $G_0, A, B$ be defined with respect to $s, \delta$ and $d$ as in the statement of \cref{lem:bipartite-case} and fix $\displaystyle r \leq \min\Big\{d^{\delta/4}, \frac{d^{\varepsilon/2}}{4cs^4}\Big\}$.
    Let $\varepsilon=\varepsilon_H$ and $c = c_H$.
    We start by ``cleaning'' $G_0$ so that no vertex in $A$ has too large degree or too low degree with respect to $B$.
    Let ${A^+ = \{a \in A \ | \ d_B(a) \geq 10d\}}$. Then, ${e(A^+, B) \geq 10d|A^+|}$.
    Since $G_0$ is $d$-degenerate, we also have ${e(A^+, B) \leq d(|A^+| + |B|)}$, therefore ${|A^+| \leq |B|}$, hence ${e(A^+, B) \leq 2d|B|}$.
    Let $A^{-} = \{a \in A \ | \ d_B(a) \leq \sqrt{d}\}$. Then, $e(A^{-}, B) \leq \sqrt{d}|A^{-}| \leq \sqrt{d}|A|$. 
    Let $A_1 = A \setminus (A^+ \cup A^{-})$.
    We have 
    $e(A_1, B) = e(A, B) - e(A^+, B) - e(A^-, B)
     \geq 3\sqrt{d}|A| - 2d|B| - \sqrt{d}|A|= 2\sqrt{d}|A| - 2{d}|B|$. Property \ref{f3:log20} implies
     $2^{d^{\delta}} \geq 2\sqrt{d}$, thus ${\sqrt{d} \cdot |A| \geq \sqrt{d} \cdot 2^{d^{\delta}}|B| \geq 2d|B|}$.
     Hence, we obtain that $e(A_1, B) \geq 2\sqrt{d}|A| - 2d|B| \geq \sqrt{d}|A|$.
    Furthermore, for every $a \in A_1$, we have $d_B(a) \in [\sqrt{d}, 10d]$. 
    Thus, $|A_1| \geq e(A_1, B)/(10d) \geq |A|/(10\sqrt{d})$.

    Since $G_0$ is $d$-degenerate, $\chi(G_0) \leq d+1$ so there exists an independent set $A_0 \subseteq A_1$ of size $|A_0|\ge |A_1|/(d+1) \geq |A_1|/(2d) \geq |A|/(20d\sqrt{d}) \geq 2^{d^{\delta}}|B| / (20d\sqrt{d})\ge 2^{2d^{\delta / 2}} |B|$ where the last inequality follows from \ref{f3:log20}.

    Let $G = G_0[A_0 \cup B]$. Since $G_0$ is $d$-degenerate, there exists an orientation $D$ of $G_0[B] = G[B]$ such that $|N_D^+(b)| \leq d$ for every $b \in B$ ($N_D^+(b)$ represents the set of out-neighbours of $b$ with respect to $D$). 
    Let $B_0$ be a random subset of $B$ containing each vertex independently with probability $1/d^2$.
    Let $B' = \{b \in B_0, N_D^+(b) \cap B_0 = \emptyset\}$. Then, $B'$ is an independent set of $G_0$.

    Let $a \in A_0$. 
    Let $d_a$ denote $|N_G(a)|$ ($= |N_B(a)|$). Then, $\sqrt{d} \leq d_a \leq 10{d}$ since $a \in A_0 \subseteq A_1$ and $A_0$ is an independent set. Since $G[N_G(a)]$ contains no $K_{s, s}$ subgraph and no induced copy of $H$, it has average degree at most $2cs^4(d_a)^{1-\varepsilon}$ by \cref{thm:induced-turan-type}.
    It is a classical result that every graph on $n$ vertices with average degree $q$ contains an independent set of size at least $n / ({2q})$ (see eg. \cite[{Theorem 3.2.1}]{alon2016probabilistic}).
    Thus, $G[N_G(a)]$ contains an independent set of size at least $\displaystyle \frac{d_a}{4cs^4(d_a)^{1-\varepsilon}} = \frac{(d_a)^{\varepsilon}}{4cs^4} \geq \frac{d^{\varepsilon/2}}{4cs^4} \geq r$ by definition of $r$. 
    For every $a \in A_0$, we fix an independent set $I_a$ of size $r$ in $G[N_G(a)]$ for the rest of the proof.

    For $a \in A_0$, let $\mathcal{E}_a$ be the event that $[N_G(a) \cap B' = N_G(a) \cap B_0 = I_a]$.
   Let $X_a = N_G(a) \setminus I_a$ and let $Y_a = \bigcup_{b \in I_a}N_D^+(b)$. 
    Note that $\mathcal{E}_a$ happens if and only if all elements of $I_a$ are put in $B_0$ and no element from $X_a 
    \cup Y_a$ is put in $B_0$.
    Recall that $d_a \leq 10d$.
    Hence, $|X_a \cup Y_a| \leq 10d + rd \le (10+d^{\delta/4})d \leq d^2/2$ where the last inequality follows from property \ref{f3:halfd}.
    Hence, using the fact that for every $x \geq -1$ and $m \geq 1$, we have $(1+x)^m \geq 1+mx$, we get: \begin{align*} 
        \mathbb{P}[\mathcal{E}_a] = p^{|I_a|}(1-p)^{|X_a \cup Y_a|} \geq \frac{1}{d^{2r}}\left(1 - \frac{1}{d^2}\right)^{d^2/2} \geq \frac{1}{2d^{2r}}\geq 2^{-1 -2r\log d} \geq 2^{-1 -2d^{\delta / 4} \log d} \geq 2^{-d^{\delta/2}}.
    \end{align*}
    Here, the last inequality follows from property \ref{f3:manyexponents}. 
    Set $A' = \{a \in A_0, \mathcal{E}_a \text{ holds}\}$.
    Then, $$\mathbb{E}[|A'|] \geq |A_0|2^{-d^{\delta/2}} \geq 2^{2d^{\delta/2}}\cdot 2^{-d^{\delta/2}}|B| = 2^{d^{\delta/2}}|B| \geq 2^{d^{\delta/2}}|B'|.$$
    
    Hence, there exists a choice of $A', B'$ satisfying the desired properties. This concludes the proof. 
\end{proof}

We are now ready to prove the following. It summarizes our current progress in both outcomes of \Cref{lem:regular-or-bipartite} and as we will see immediately after it strengthens and implies \cref{thm:McCarty}.

\begin{theorem}\label{thm:bipartite-degree-boundedness}
    Let $k$ be a positive integer. There exists $\varepsilon = \varepsilon(k) > 0$ and a polynomial $P_k$ with the following property.
    For any two positive integers $s$ and $d \ge P_k(s)$,  
    every $K_{s,s}$-subgraph-free graph with maximum average degree $d$ contains either a $C_4$-subgraph-free bipartite induced subgraph with average degree at least $k$, or an induced bipartite subgraph $G'=(A',B')$ with $|A'| \geq 2^{d^{2\varepsilon}} |B'|$, and $|N(a) \cap B'| = \lfloor d^\varepsilon/s^4 \rfloor$ for every $a \in A'$.
\end{theorem}

\begin{proof}
    Let $\varepsilon_k,p_k(s)$ be provided by \Cref{lem:balanced-case-to-bipartite}. 
    Let $H$ be a fixed $C_4$-subgraph-free bipartite graph with average degree at least $k$ (for instance the incidence graph of a projective plane of order at least $k$).  We set $\varepsilon=\varepsilon(k)=\min\{\varepsilon_k/40, \varepsilon_H/3\}$.
    Let $c = c_H$, $\gamma = \varepsilon_k/8$, $\delta = \varepsilon_k/10$ and let $C$ be a sufficiently large constant (depending only on $k$) so that $C \ge f_1(\gamma), f_3(\delta), 144,(4c)^{6/\varepsilon_H}$ and so that $x^{\gamma} \ge x^{\delta}+2$ for any $x \ge C$. (Recall, $f_1, f_3$ are defined in \cref{lem:regular-or-bipartite} and in \cref{lem:bipartite-case}, respectively.)
    We set $\displaystyle P_k(s) \coloneqq |p_k(s)|^{1/(1-5\gamma)}+C.$

The result is trivial for $s < 2$.
Let $s \geq 2, d \geq P_k(s)$ and $G$ be a graph with maximum average degree $d$ that does not contain $K_{s, s}$ as a subgraph. If $G$ contains $H$ as an induced subgraph we are done, so we may assume it does not.

Let $G_0$ be an induced subgraph of $G$ achieving maximum average degree, i.e. $d(G_0) = d \geq P_k(s)$. $G_0$ is $d$-degenerate by maximality of $d$. 
Note that $\gamma < 1/5$ and $d \geq f_1(\gamma)$.
By applying \cref{lem:regular-or-bipartite} we obtain that either \ref{regular}
or \ref{bipartite} holds for $\gamma, d$.

If case \ref{regular} occurs, $G_0$ contains an induced subgraph $G^*$ of average degree at least $6d^{1-5\gamma}$ and maximum degree $\Delta(G^*) \leq 6d^{1+3\gamma}$. So, $d(G^*) \ge 6d^{1-5\gamma} \ge (6d^{1+3\gamma})^{1-\varepsilon_k}\ge (\Delta(G^*))^{1-\varepsilon_k}$ where the second inequality follows from the fact that $\gamma < \frac{1}{5}$.
Moreover, $\Delta(G^*) \ge 6d^{1-5\gamma} \ge p_k(s)$. 
So we may apply \Cref{lem:balanced-case-to-bipartite} with $\Delta = \Delta(G^*)$ to find a $C_4$-subgraph-free bipartite induced subgraph with average degree at least $k$, as desired.

If case \ref{bipartite} occurs, there is a partition $V(G_0) = A \sqcup B$ such that $e(A, B) \geq |V(G_0)|d/4 \geq |A|d/4 \ge 3\sqrt{d}|A|$ 
and $|A| \geq 2^{d^{\gamma}-2}|B| \ge 2^{d^{\delta}}|B|$.
Since $G_0$ is $d$-degenerate and has no $H$ as an induced subgraph or $K_{s,s}$ as a subgraph, and $\displaystyle r \coloneqq \lfloor d^\varepsilon/s^4 \rfloor \le \min\Big\{d^{\delta/4}, \frac{d^{\varepsilon_H/2}}{4c_Hs^4}\Big\}$ (using that $\varepsilon \leq \varepsilon_k/40 = \delta/4$, and $d \geq P_k(s) \geq C \geq (4c_H)^{6/\varepsilon_H}$ so $d^{\varepsilon} \leq d^{\varepsilon_H/3} = d^{\varepsilon_H/2}/d^{\varepsilon_H/6} \leq d^{\varepsilon_H/2}/(4c_H)$) we may apply \Cref{lem:bipartite-case} to find subsets $A' \subseteq A, B' \subseteq B$ that are both independent sets in $G_0$, with $|A'| \geq 2^{d^{\delta/2}}|B'|\ge 2^{d^{2\varepsilon}}|B'|$ and $|N(a) \cap B'| = \lfloor d^\varepsilon/s^4 \rfloor$ for every $a \in A'$, as desired.
\end{proof}

Now, we quickly show that \cref{thm:bipartite-degree-boundedness} implies \cref{thm:McCarty}, thus answering McCarty's \cite{rose-matrix} question on polynomial degree bounding functions.

\begin{proof}[Proof of \cref{thm:McCarty}.]
    Let $\mathcal{G}$ be a hereditary class of graphs for which there is a polynomial $p$ such that every \emph{bipartite} $K_{s,s}$-subgraph-free graph in $\mathcal{G}$ has average degree at most $p(s)$.
    Let $k=p(2)+1$, and let $\varepsilon_k>0, P_k$ be given by \cref{thm:bipartite-degree-boundedness}.
    Set $p'(s)=(2s^4p(s))^{1 / \varepsilon_k} + P_k(s)$.
    Towards a contradiction, suppose there exists a $K_{s,s}$-subgraph-free graph $G \in \mathcal{G}$ with average degree more than $p'(s)$. Then, $G$ has maximum average degree $d > p'(s)$. Since $k > p(2)$ and $\mathcal{G}$ is hereditary, $G$ cannot contain a $C_4$-subgraph-free bipartite induced subgraph with average degree at least $k$. Then, by \cref{thm:bipartite-degree-boundedness}, $G$ contains an induced bipartite subgraph $G'=(A',B')$ with $|A'| \geq 2^{d^{2\varepsilon_k}}|B'| \geq |B'|$ and $|N(a) \cap B'| = \lfloor d^{\varepsilon_k}/s^4 \rfloor$ for every $a \in A'$. Note that $G' \in \mathcal{G}$ and $G'$ is $K_{s, s}$-subgraph-free. Furthermore, $e(G') = \lfloor d^{\varepsilon_k}/s^4 \rfloor \cdot |A'| \geq \frac{1}{2}\lfloor d^{\varepsilon_k}/s^4 \rfloor \cdot |A' \cup B'|$ so $G'$ has average degree at least $\lfloor d^{\varepsilon_k}/s^4 \rfloor > p(s)$, a contradiction.
\end{proof}

\section{Induced subdivisions}\label{sec:consequence}

In this section, we prove \cref{thm:poly-kuhn-osthus} and \cref{thm:chi-bound}.
We begin with \cref{thm:poly-kuhn-osthus}.
First, we need to gather a few more well-known results.

The outcome of \cref{thm:bipartite-degree-boundedness} in which we find an induced bipartite $C_4$-subgraph-free subgraph with large average degree can be handled by the special case ($s=2$) of Theorem 1 of Kühn and Osthus \cite{Kuhn_Osthus00} stating that $C_4$-subgraph-free graphs containing no induced subdivision of a graph $H$ have bounded average degree.

\begin{lemma}[\protect{Theorem 1 of \cite{Kuhn_Osthus00}}] \label{lem:kuhn-osthus-C4}
    Let $H$ be a graph. There exists $k=k(H)$ such that any $C_4$-subgraph-free graph containing no induced subdivision of $H$ has average degree less than $k$.
\end{lemma}

To handle the very unbalanced case of \cref{thm:bipartite-degree-boundedness}, we shall also take inspiration from \cite{Kuhn_Osthus00}.
We will need the following classical result of Kühn and Osthus. Roughly speaking, it states that if we have an unbalanced bipartite graph $G'$ where each vertex on the ``big'' side of $G'$ has large degree, then $G'$ contains a large complete bipartite subgraph or an induced 1-subdivision of some graph of large average degree.

\begin{lemma}[\protect{Corollary 19 of \cite{Kuhn_Osthus00}}] \label{lem:kuhn-osthus}
  Let $\alpha \in \mathbb{N}$ and $\beta \geq 8\alpha$. Let $G' = (A', B')$ be a bipartite graph such that {$|A'| \geq \beta^{12\alpha}|B'|$} and $\beta/4 \leq d(a) \leq 4\beta$ for every $a \in A'$. Then, $G'$ contains either a $K_{\alpha, \alpha}$-subgraph or an induced 1-subdivision of some graph $F$ with {$d(F) \geq \beta^9/2^{14}$.}
\end{lemma}

Finally, we will use the fact that every graph of large average degree contains a subdivision of a large complete graph. The following result was proved independently by Bollob\'as-Thomason \cite{BT98} and Koml\'os-Szemer\'edi \cite{KS96}\footnote{Technically speaking, the bound obtained in this paper has a worse constant factor in front of $h^2$.}. The latter of these was one of the earliest papers introducing the highly influential concept of sublinear expander graphs, used, for example, by Montgomery to essentially resolve the famous Ryser-Brualdi-Stein conjecture \cite{montgomery2023proof}.

\begin{lemma}\label{lem:degree-subdiv-Kk}
    For all $h \in \mathbb{N}$, if $G$ is a graph with average degree at least $256h^2$, then $G$ contains a subdivision of $K_h$.
\end{lemma}

We are now ready to prove \cref{thm:poly-kuhn-osthus}, stating that $K_{s,s}$-subgraph-free graphs with no induced subdivision of a fixed graph $H$ have at most polynomial average degree (in terms of $s$).

\begin{proof}[Proof of \cref{thm:poly-kuhn-osthus}.] 
    By \cref{lem:kuhn-osthus-C4}, there exists a positive integer $k$ such that every $C_4$-subgraph-free graph not containing an induced subdivision of $H$ has average degree less than $k$. Let $\varepsilon=\varepsilon(k)$ and $P_k(s)$ be provided by \Cref{thm:bipartite-degree-boundedness}. 
    Let $p(s) \coloneqq |P_k(s)|+(2|H|s)^{5/\varepsilon}+C$, where $C$ is a constant depending only on $k$, chosen large enough to guarantee that $d^{\varepsilon} \ge \frac{3}{2} \varepsilon \log d$ whenever $d \geq C$.
    We claim that $p(s)$ satisfies the conditions of \cref{thm:poly-kuhn-osthus}.

    Let $s \geq 1$ be a fixed integer. 
    If $s=1$ the result is immediate so we may assume $s \ge 2$.
    Let $G$ be a $K_{s,s}$-subgraph-free graph containing no induced subdivision of $H$ that has average degree more than $p(s)$. Then, $G$ has maximum average degree $d > p(s)$.
    Since $G$ does not contain an induced $C_4$-subgraph-free subgraph with average degree at least $k$ by \cref{lem:kuhn-osthus-C4}, it follows from \cref{thm:bipartite-degree-boundedness} that $G$ contains an induced bipartite subgraph $G'=(A',B')$ with $|A'| \geq 2^{d^{2\varepsilon}}|B'|$ and $|N(a) \cap B'| = \lfloor d^\varepsilon/s^4 \rfloor$ for every $a \in A'$.

    Let $\alpha = \lfloor \lfloor d^\varepsilon/(s^4) \rfloor / 8 \rfloor \ge s$ and let $\beta = \lfloor d^\varepsilon/(s^4) \rfloor \geq 8 \alpha$.
    Note that by our choice of $C$ we have $\beta^{12\alpha} \le \left(d^{\varepsilon}\right)^{12d^{\varepsilon}/8} \le 2^{d^{2\varepsilon}}$.
    Since $G'$ is $K_{s,s}$-subgraph-free, by \cref{lem:kuhn-osthus}, $G'$ contains an induced 1-subdivision $F'$ of some graph $F$ with average degree 
    $d(F) \geq \beta^9/2^{14}\ge 256 |H|^2$, where in the last inequality we used $\beta \ge (2|H|s)^5/(2s^4) \ge 2^4|H|$. 
    By \cref{lem:degree-subdiv-Kk}, $F$ contains as a subgraph a subdivision of $K_{|H|}$. 
    In $F'$ (hence in $G'$ and thus in $G$), this corresponds to an induced proper subdivision of $K_{|H|}$. 
    Since every proper subdivision of $K_n$ contains an \emph{induced} subdivision of any $n$-vertex graph, this completes the proof. 
\end{proof}

Now, we prove \cref{thm:chi-bound}, i.e. that the class of graphs containing no induced subdivision of $K_{r,t}$ is polynomially $\chi$-bounded. In fact, we shall prove that the average degree of graphs containing no induced subdivision of $K_{r,t}$ is at most polynomial in their clique number, which clearly implies \cref{thm:chi-bound} since $d$-degenerate graphs are $(d+1)$-colorable.
This is a simple application of \cref{thm:poly-kuhn-osthus} and the classical bound on the Ramsey numbers by Erd\H{o}s and Szekeres~\cite{ES1935}.

\begin{theorem}
    Let $r,t$ be positive integers. Then there exists a polynomial $g$, such that the average degree of every graph $G$ containing no induced subdivision of $K_{r,t}$ is at most $g(\omega(G))$.
\end{theorem}
\begin{proof}
    Assume $s \leq t$. 
    By \cref{thm:poly-kuhn-osthus}, there exists a polynomial $p$ such that for every integer $x$ any graph $G$ that does not contain $K_{x, x}$ as a subgraph nor an induced subdivision of $K_{r, t}$ has average degree at most $p(r)$.
    
    Let $k \geq 1$ and set $x \coloneqq (k+t-1)^{t-1} \geq \binom{k+t-1}{t-1}$.
    Let $g(k) \coloneqq p(x)$. Then, $g$ is a polynomial in $k$. 
    Let $G$ be a graph with $\omega(G) = k$ that does not contain an induced subdivision of $K_{r, t}$.

    Suppose $G$ contains a $K_{x, x}$ as a subgraph and let $X, Y \subseteq V(G)$ be the corresponding vertex sets.
    By the bound on Ramsey numbers by Erd\H{o}s and Szekeres~\cite{ES1935}, every graph on $r$ vertices contains a clique of size $k+1$ or a stable set of size $t$.
    Since $\omega(G) = k$ it follows that $X$ and $Y$ both contain a stable set of size $t$.
    But then $G$ contains an induced $K_{r,t}$, a contradiction.
    Hence, $G$ contains no $K_{x, x}$, so by \cref{thm:poly-kuhn-osthus}, $G$ has average degree at most $p(x) = g(k)$. 
\end{proof}


\section*{Acknowledgments} 
The bulk of this work was done at the 2023 Structural Graph Theory workshop held at the IMPAN B\k{e}dlewo Conference Center (Poland) from September 27- October 2. We announced some of our results and methodology at the progress reports in this workshop and gratefully acknowledge the encouragement of the other participants, especially Am\'{a}lka Masa\v{r}ikova.   
This workshop was a part of STRUG: Stuctural Graph Theory Bootcamp, funded by the ``Excellence initiative – research university (2020-2026)" of University of Warsaw.
Some of our initial discussions about this problem were held at the MATRIX-IBS Workshop: Structural Graph Theory Downunder III in April, 2023.
We gratefully acknowledge the support from MATRIX, IBS and University of Warsaw for making this work possible.

The first author would like to thank Marcin Pilipczuk for inviting him and covering his travel expenses for the 2023 Structural Graph Theory workshop. The second author gratefully acknowledges the support of the Oswald Veblen Fund. 
The third author was supported by the Institute for Basic Science (IBS-R029-C1).

We are also grateful to the anonymous referees for their suggestions on improving the presentation.

\bibliographystyle{amsplain}



\begin{aicauthors}
\begin{authorinfo}[rb]
 
  Romain Bourneuf\\
  LaBRI, Université de Bordeaux\\
  Bordeaux, France \\
  \url{https://perso.ens-lyon.fr/romain.bourneuf/} \\
\end{authorinfo}
\begin{authorinfo}[mb]
  Matija Buci\'c\\
  Princeton University\\
  Princeton, USA\\
  mb5225\imageat{}princeton\imagedot{}edu \\
\url{https://sites.google.com/princeton.edu/matija-bucic}
\end{authorinfo}
\begin{authorinfo}[lc]
  Linda Cook\\
  Institute for Basic Science\\ 
  Daejeon, South~Korea\\
  \url{https://dimag.ibs.re.kr/home/cook/}\\
  Current affiliation: KdVI, University of Amsterdam, The Netherlands 
\end{authorinfo}
\begin{authorinfo}[jd]
  James Davies\\
  University of Cambridge\\ 
  Cambridge, UK\\  
  \url{https://sites.google.com/view/davies-james/}\\
\end{authorinfo}
\end{aicauthors}

\end{document}